\documentclass[12pt,reqno]{amsart}

\usepackage[all]{xy}
\usepackage{amsmath}
\usepackage{amsfonts}
\usepackage{amssymb}
\usepackage[all]{xy}
\usepackage{mathrsfs}
\usepackage{bbding}
\usepackage{txfonts}
\usepackage{amscd}

\usepackage[shortlabels]{enumitem}
\usepackage{ifpdf}
\ifpdf
\usepackage[colorlinks,final,backref=page,hyperindex]{hyperref}
\else
\usepackage[colorlinks,final,backref=page,hyperindex,hypertex]{hyperref}
\fi
\usepackage{tikz}
\usepackage[active]{srcltx}

\def\W{\mathcal{SD}}

\def\M{M}
\def\Z{\mathbb{Z}}
\def\N{\mathbb{N}}

\def\C{\mathbb{C}}

\numberwithin{equation}{section}
\newtheorem{theo}{Theorem}[section]
\newtheorem{defi}[theo]{Definition}
\newtheorem{coro}[theo]{Corollary}
\newtheorem{lemm}[theo]{Lemma}
\newtheorem{prop}[theo]{Proposition}

\newtheorem{case}{Case}

\newtheorem{rema}[theo]{Remark}

\newcommand{\nc}{\newcommand}

\nc{\tred}[1]{\textcolor{red}{#1}}
\nc{\tblue}[1]{\textcolor{blue}{#1}}
\nc{\tgreen}[1]{\textcolor{green}{#1}}
\nc{\tpurple}[1]{\textcolor{purple}{#1}}
\nc{\btred}[1]{\textcolor{red}{\bf #1}}
\nc{\btblue}[1]{\textcolor{blue}{\bf #1}}
\nc{\btgreen}[1]{\textcolor{green}{\bf #1}}
\nc{\btpurple}[1]{\textcolor{purple}{\bf #1}}

\nc{\yf}[1]{\textcolor{red}{Yufeng:#1}}
\nc{\ld}[1]{\textcolor{blue}{Liu Dong: #1}}
\nc{\hb}[1]{\textcolor{purple}{Haibo: #1}}

\begin{document}

\title[Irreducible    modules for  N=2 superconformal algebras]{Irreducible    modules over  N=2  superconformal algebras from algebraic D-modules}

\author{Haibo Chen}
\address{Department of Mathematics, Jimei University, Xiamen, Fujian 361021, China}
\email{hypo1025@jmu.edu.cn}

\author{Xiansheng Dai}
\address{School of Mathematical Sciences, Guizhou Normal University,
Guiyang 550001,  China}
\email{daisheng158@126.com}

\author{Dong Liu}
\address{Department of Mathematics, Huzhou University, Zhejiang Huzhou, 313000, China}
\email{liudong@zjhu.edu.cn}

\author[Yufeng Pei]{Yufeng Pei}
\address{Department of Mathematics, Huzhou University, Zhejiang Huzhou, 313000, China}\email{pei@zjhu.edu.cn}

\date{\today}

\vspace{-5mm}

\begin{abstract}
 In this paper, we introduce a family of functors denoted $\mathscr{F}_b$ that act on algebraic D-modules and generate modules over N=2 superconformal algebras. We prove these functors preserve irreducibility for all values of $b$, with a few clear exceptions described. We also establish necessary and sufficient conditions to determine when two such functors are naturally isomorphic. Applying $\mathscr{F}_b$ to  N=1 super-Virasoro algebras recovers the functors previously introduced in \cite{CDLP}. Our new functors also facilitate the recovery of specific irreducible modules over N=2 superconformal algebras, including intermediate series and $U(\mathfrak{h})$-free modules. Additionally, our constructed functors produce several new irreducible modules for N=2 superconformal algebras.

\end{abstract}


.


\maketitle
\vspace{-10mm}

\section{Introduction}

 N=2 superconformal algebras play an important role in  two-dimensional quantum field theory and conformal field theory models. This area of study has received extensive attention from both physicists and mathematicians. The four sectors that comprise  N=2 superconformal algebras are the Neveu-Schwarz sector, Ramond sector, topological sector, and twisted sector. Interestingly, Schwimmer and Seiberg discovered that the Neveu-Schwarz sector and Ramond sector are isomorphic \cite{SS}. Additionally, Dijkgraaf et al. introduced the topological sector through a topological twist in their work \cite{DVV}. Remarkably, the topological sector is isomorphic to the Neveu-Schwarz sector, resulting in only two distinguishable sectors in  N=2 superconformal algebras.

Significant progress has been made in  the weight representation theory  of  N=2 superconformal algebras \cite{FST,FSST,IK,IO1,IO2,KS,ST}. Furthermore, researchers explored the representation theory of  N=2 superconformal algebras from a vertex algebra perspective, as demonstrated in works such as \cite {Ad1,KW,Sa}. The classification of irreducible modules with finite-dimensional weight spaces for the  sector has been successfully accomplished  recently in \cite{LPX0,XL}. There are three types of modules: irreducible highest weight or lowest weight modules, and irreducible modules of the intermediate series (see \cite{FJS,LSZ}). Additionally, a family of non-weight modules for N=2  superconformal algebras, known as $U(\mathfrak{h})$-free modules,  have been constructed and classified in \cite{CDL,YYX2}.

The theory of modules over Weyl algebras, also known as algebraic D-modules, is a vital tool for studying the representation theory of Lie (super)algebras. Utilizing this theory, researchers have constructed and analyzed many fascinating irreducible modules, such as intermediate series modules and $U(\mathfrak{h})$-free modules \cite{CCGMZ,Nil}, for the Virasoro algebra \cite{LGZ,LZ,LZ0}  and N=1 super-Virasoro algebras \cite{CDLP}.

This paper aims to extend D-module theory to representations of  N=2 superconformal algebras. We introduce a family of functors, denoted $\mathscr{X}_b$, which act on Weyl superalgebra modules and produce modules for  N=2 superconformal algebras. These functors depend on a complex parameter $b\in \mathbb{C}$. By applying the functor $\mathscr{S}$ from \cite{CDLP}, mapping Weyl algebra modules to Weyl superalgebra modules, we obtain another family functor $\mathscr{F}_b$ that produces  modules of  N=2 superconformal algebras from algebraic D-modules (Theorem \ref{M}). We show these functors preserve irreducibility unless $b=0$ or $\frac{1}{2}$ (Theorem \ref{M1}). We also determine the necessary and sufficient conditions for two such functors to be naturally isomorphic (Theorem \ref{th4333}).  Restricting $\mathscr{F}_b$ to   N=1 super-Virasoro algebras recovers the functor introduced in \cite{CDLP}. Our functors also facilitate retrieving certain known  irreducible modules over N=2 superconformal algebras, like intermediate series and  $U(\mathfrak{h})$-free ones \cite{YYX1}. Furthermore, we present new examples of irreducible  modules over  N=2 superconformal algebras through our constructed functors.

Expanding the irreducible module construction techniques introduced in this work to encompass a wider range of superconformal algebra classes remains an open challenge. We also aim to further explore the generalization of these methods in future work.

The paper is structured as follows:
Section 2 reviews N=2 superconformal algebras, the Weyl algebra, and the Weyl superalgebra.
In Section 3, we introduce a family of functors that map Weyl algebra modules to  modules over N=2  superconformal algebras.
In Section 4, we examine the properties of irreducibility preservation and investigate natural isomorphisms for these functors.
Finally, in Section 5, we utilize these functors to construct  several  new irreducible modules over N=2  superconformal algebras.

Throughout this paper, $\C,\C^*,\Z, \Z^*, \N$ and $\Z_+$ will denote the sets of complex numbers, nonzero complex numbers, integers, nonzero integers, nonnegative integers, and positive integers, respectively. All vector spaces, vector algebras, and modules will be considered to be over the field of complex numbers $\C$, and all modules for Lie superalgebras will be $\Z_2$-graded.

\section{Preliminary}

This section reviews fundamental definitions, notations, and results related to  N=2 superconformal algebras and Weyl superalgebras.

\subsection{Some basic definitions  of superalgebras}

Let ${L}$ be a Lie algebra or a Lie superalgebra. We will use ${U}(L)$ to refer to its universal enveloping algebra. The Kronecker delta will be denoted by $\delta_{i,j}$.
For any module $V=V_{\bar 0}\oplus V_{\bar1}$ over a Lie superalgebra, the parity change functor will be denoted by $\Pi(V)$. It is defined as  $(\Pi(V))_{\bar0}= V_{\bar 1}$ and $(\Pi(V))_{\bar1}=V_{\bar0}$.
Suppose $L$ is an associative (super)algebra or a Lie (super)algebra. We will denote the category of $L$-modules as ${\rm Mod}\ L$.

Let $M$ be a module over an associative (super)algebra $L$. We say that the action of $x\in L$ on $v\in M$ is nilpotent if there exists $n\in\Z_+$ such that $x^n v=0$ for $v\neq 0$.
Let $E$ be a subspace of $L$.  If there exists $0\neq v\in M$ and $x\in E$ such that $xv=0$, then we say that $M$ is $E$-torsion. Otherwise, we say that $M$ is $E$-torsion-free.


\subsection{$N=2$ superconformal   algebras}

\begin{defi}
  {The $N=2$ superconformal algebra} $\hat{\mathcal{G}}[\epsilon](\epsilon=0,\frac{1}{2})$ is an infinite-dimensional Lie superalgebra. It has a basis of even elements given by $L_n$, $H_n$, $C$, where $n\in \mathbb{Z}$, and a basis of odd elements given by $G^+_p$ and $G^-_p$, where $p\in \mathbb{Z}+\epsilon$. The algebra is defined by the following relations:
 \begin{equation*}\label{def2131.1}
\aligned
&[L_m,L_n]= (m-n)L_{m+n}+\frac{m^{3}-m}{12}\delta_{m+n,0}C,\\
&[H_m,H_n]=\frac{m}{3}\delta_{m+n,0}C,\\  &[L_m,H_n]=-nH_{m+n},\\
& [H_m,G_{p}^{\pm}]=\pm G_{m+p}^{\pm}, \\
&[L_m,G_{p}^{\pm}]=\left(\frac{m}{2}-p\right)G_{m+p}^{\pm},\\
&[G_{p}^-,G_{q}^+]= 2L_{p+q}-(p-q)H_{p+q}+\frac{4p^2-1}{12}\delta_{p+q,0}C,\\&
[G_{p}^{+},G_{q}^{+}]=[G_{p}^{-},G_{q}^{-}]=[\hat{\mathcal{G}}[\epsilon],C]=0,
\endaligned
\end{equation*}
  where $m,n\in\Z$, $p,q\in\Z+\epsilon$. ${\hat{\mathcal{G}}}[0]$  is called the $N=2$ Ramond algebra and
$\hat{\mathcal{G}}[\frac{1}{2}]$ is called
the  $N=2$ Neveu-Schwarz algebra.
  \end{defi}

\begin{lemm}[ \cite{SS}]\label{lemm2.5}
The $N=2$ Neveu-Schwarz algebra $\hat{\mathcal{G}}[{\frac{1}{2}}]$ and the $N=2$ Ramond algebra $\hat{\mathcal{G}}[0]$ are isomorphic to each other by the spectral shift isomorphism:
$\delta: \hat{\mathcal{G}}[{\frac{1}{2}}] \rightarrow \hat{\mathcal{G}}[{0}]$, defined by
\begin{eqnarray*}
&&\delta(L_m)=L_m+\frac{1}{2} H_m+\delta_{m,0}\frac{C}{24},
\\&& \delta(H_m)=H_m+\delta_{m,0}\frac{C}{6},
\\&&
\delta(G_p^{\pm})=G_{p\pm\frac{1}{2}}^{\pm},\\
&&\delta(C)=C,
\end{eqnarray*}
where $m\in\Z,p\in\frac{1}{2}+\Z$.

\end{lemm}

To simplify notation, we denote
$$
\hat{\mathcal{G}}=\hat{\mathcal{G}}[0] \cong \hat{\mathcal{G}}[\frac{1}{2}],\quad \mathcal{G}=\hat{\mathcal{G}}/\C C.
$$

 \subsection{Weyl (super)algebra}
Denote by
$
\mathcal{A}=\mathbb{C}\left[t, t^{-1}, \theta\right]
$
the associative superalgebra of Laurent polynomials associated to an even, $t$, and an odd, $\theta$, formal variable. Assume that $\theta^2=0$ and $\theta t=t \theta$.   Let  $ \partial_t=\frac{d}{dt}$, $D=t\frac{d}{dt}$ and $\partial_\theta =\frac{d}{d\theta}$.

\begin{defi}
The Weyl algebra $\mathcal{D}$ is the associative algebra of regular differential operators on the circle $S^1$generated by  $t^{\pm1}$ and  $\partial_t$,  subject to the relations  $\partial_t t - t \partial_t = 1.$
\end{defi}

 Furthermore, all irreducible $\mathcal{D}$-modules  have been classified.

\begin{lemm}[\cite{B,LZ}]\label{211}
	Let $M$ be an  irreducible  $\mathcal{D}$-module.
	\begin{itemize}
		\item[\rm(i)]  If $M$ is $\C[t^{\pm1}]$-torsion-free, then $M\cong \mathcal{D}/(\mathcal{D}\cap\big(\C(t)[ D]\tau)\big)$ for some irreducible element   $\tau$    in the associative algebra $\C(t)[ D]$;
		\item[\rm(ii)]  If $M$ is $\C[t^{\pm1}]$-torsion, then $M\cong \Omega(\lambda)$ for some $\lambda\in\C^*$, where $\Omega(\lambda)=\C [ D]$ is an irreducible $\mathcal{D}$-module with
  $$
  t^m D^n=\lambda^m( D-m)^n,\quad  D D^n= D^{n+1},\quad\forall n\in\N,m\in\Z.
  $$
	\end{itemize}
\end{lemm}

\begin{defi}[\cite{CW}]
The Weyl superalgebra $\mathcal{SD}$ is  the associative superalgebra of regular differential operators
on the supercircle $S^{1 \mid 1}$  generated by  $t^{\pm1}$ and $\theta$, $\partial_t$, $\partial_{\theta}$, subject to the relations
$$\partial_t t - t \partial_t = 1, \quad \partial_{\theta} \theta + \theta \partial_{\theta} = 1, \quad \partial_t\theta = \partial_{\theta} t = 0.$$
\end{defi}

Note that the following elements
$$
t^k D^l, t^k D^l \theta \partial_\theta,t^k D^l \theta, t^k D^l \partial_\theta,\quad  k \in \mathbb{Z}, l \in \mathbb{N}
$$
form a linear basis of $\mathcal{S D}$, where   the odd elements $\theta$ and $\partial_\theta$ generate the four-dimensional Clifford superalgebra with relation $\theta \partial_\theta+\partial_\theta \theta=1$. Clearly, $\mathcal{SD} $ is isomorphic to the tensor product of the Weyl algebra $\mathcal{D}$ and the Clifford superalgebra.

In \cite{CDLP}, we constructed a functor $\mathscr{S}:{\rm Mod}\,\mathcal{D}\to {\rm Mod}\,\mathcal{SD}$ that preserves irreducibility and establishes a one-to-one correspondence between irreducible modules over the Weyl algebra and the Weyl superalgebra.

\begin{prop}[\cite{CDLP}]\label{S}Let $M$ be a $\mathcal{D}$-module, and let us take a copy of $M$, which we will denote as $\overline M$. We can then turn this copy into a $\mathcal{D}$-module using the operation
\begin{equation}
    x\cdot \overline{v}=\overline{xv},\quad \forall x\in \mathcal{D},\overline v\in \overline M.\label{S1}
\end{equation}
 Define
\begin{equation}
  \partial_\theta \cdot v=0,\quad \theta\cdot v=\overline v,\quad \partial_\theta\cdot \overline v=v,\quad  \theta\cdot \overline v=0,\quad \forall v\in M. \label{S2}
\end{equation}
 Then
 \begin{itemize}
     \item[\rm (i)]  $\mathscr{S}(M)=M\oplus \overline M$ is an $\mathcal{SD}$-module,
     \item[\rm (ii)]there exists a functor $\mathscr{S}:{\rm Mod}\,\mathcal{D}\to {\rm Mod}\,\mathcal{SD}$ with $M\mapsto \mathscr{S}(M)$,
     \item[\rm (iii)]the functor $\mathscr{S}$ preserves irreducible objects,
     \item[\rm(iv)]there is a one-to-one correspondence between the irreducible $\mathcal{SD}$-modules and irreducible $\mathcal{D}$-modules.
 \end{itemize}

\end{prop}

\section{Constructions of functors}

In this section, we introduce a family of functors that map modules over the Weyl algebra  to modules over the N=2  superconformal algebra.

 Let $\mathcal{SD}^-$ denote the Lie superalgebra associated to the associative superalgebra $\mathcal{SD}$ with the following bracket:
$$
[x,y]=xy-(-1)^{|x||y|}yx,\quad\forall x,y\in \mathcal{SD}.
$$
Let $\mathcal{SW}$ denote the Witt superalgebra, which is the subalgebra of $\mathcal{SD}^-$, spanned by the super differential operator of order 1:
$$L_m^\prime=t^mD,\quad  H_m^\prime= t^m\theta\partial_\theta,\quad Q_m=t^m\partial_\theta,\quad P_m= t^m\theta D,$$
with the following brackets:
\begin{eqnarray*}
&&[L_m^\prime, L_n^\prime]=(n-m)L_{m+n}^\prime,\\
&&[L_m^\prime,H_n^\prime]=nH_{m+n}^\prime,
\\
&&  [L_m^\prime,Q_n]=nQ_{m+n},\\
&& [L_m^\prime,P_n]=(n-m) P_{m+n},\\
&&[H_m^\prime,Q_n]=-Q_{m+n},
\\
&& [H_m^\prime,P_n]=P_{m+n},\\
&&[Q_m,P_n]=L_{m+n}^\prime+mH_{m+n}^\prime,
\\&&[Q_m,Q_n]=[P_m,P_n]=[H_m^\prime,H_n^\prime]=0.
\end{eqnarray*}

The following lemma can be verified easily.

\begin{lemm}[\cite{DVV}]\label{Iso}
The centerless $N=2$ superconformal algebra $\mathcal{G}$ is isomorphic to the Witt superalgebra $\mathcal{SW}$. This can be shown through the map $\varpi: {\mathcal{G}}\rightarrow {\mathcal{SW}}$, which is defined as follows:
\begin{eqnarray}\label{weq121}
&&\varpi(L_m)= -t^{m}\left( D+\frac{m}{2}\theta\partial_\theta\right),\
\varpi(H_m)=t^{m}\theta\partial_\theta,
\
\varpi(G_m^+)=-2t^{m}\theta D,\
\varpi(G_m^-)= t^{m}\partial_\theta,
\end{eqnarray}
where $m\in\Z$.
\end{lemm}

We denote the topological $N=2$ superconformal algebra by $\mathcal{T}$.
\begin{rema}
To illustrate the relationships among these Lie superalgebras, we present the following commutative diagram:
\begin{displaymath}
\xymatrix{
  \mathcal{G}[0] \ar[d]_{\delta} \ar[r]^{\varpi}
                & \mathcal{SW} \ar[d]^{\cong}  \\
  \mathcal{G}[\frac{1}{2}] \ar@{.>}[ur]|-{\varpi\circ\delta^{-1}} \ar[r]_{\varpi\circ\delta^{-1}}
                & \mathcal{T}             }
                \end{displaymath}
\end{rema}

\begin{lemm}\label{lek22}

For an $\mathcal{SD}$-module $V$ and $b\in\mathbb{C}$, we can define $V$ as a $\mathcal{G}$-module using the following actions:
\begin{eqnarray}\label{wq33}
&& L_m\cdot v=-t^{m}\big( D+mb+\frac{m}{2}\theta\partial_\theta\big)v,
\\&& \label{wq34}H_m\cdot v=t^{m}\big(-2b+\theta\partial_\theta\big)v,
\\&&\label{wq35}G_m^+\cdot v=-2t^{m}\big(\theta D+2bm\theta\big)v,
\\&&\label{wq36}G_m^-\cdot
 v=\big(t^{m}\partial_\theta\big)v,\label{wq36}
\end{eqnarray}
for $m\in\mathbb{Z}$  and $v\in V$. This $\mathcal G$-module is denoted by $\mathscr{X}_b(V)$.
\end{lemm}

\begin{proof}Consider the extended  algebra $\tilde{\mathcal{G}}= {\mathcal{G}} \ltimes \mathcal A$   with the
brackets
$$[x,y]=-(-1)^{|x||y|}[y,x]=x(y),\ [y,y^\prime]=0,\quad \forall x\in \mathcal{G},y,y^\prime\in \mathcal A.$$
By Lemma \ref{Iso},  $\tilde{\mathcal{G}}$ can be embedded into the Lie supealgebra $\mathcal{SD}^-$ associated to the associative superalgebra $\mathcal{SD}$
by the following  linear map $\Phi: \tilde{\mathcal{G}}\hookrightarrow \mathcal{SD}^-$ with
\begin{eqnarray*}
&&\nonumber L_m\longmapsto -t^{m}\big( D+\frac{m}{2}\theta\partial_\theta\big),\\&&
H_m\longmapsto t^{m}\theta\partial_\theta,
\\&&
G_m^+\longmapsto-2t^{m}\theta D,
\\&&
G_m^-\longmapsto t^{m}\partial_\theta,
\\&&  t^m\mapsto t^m,
\\&&t^m\theta\mapsto t^m\theta,
\end{eqnarray*}
where $m\in\Z.$  For $b\in \mathbb{C}$,  we define a linear map $\sigma_{b}: \tilde{\mathcal{G}}\to \tilde{\mathcal{G}}$ by
 \begin{eqnarray*}
\sigma_b(L_m)&=&-t^{m}D-\frac{m}{2}t^m\theta\partial_{\theta}-mbt^{m},\\
\label{LI4455e2.2} \sigma_b(H_m)&=&t^{m}\theta\partial_\theta-2bt^m,\\
\sigma_b(G^+_m)&=& -2t^{m}\theta D-4bmt^m\theta,\\
\sigma_b(G^-_m)&=&t^m\partial_\theta,\\
\sigma_b(t^m)&=& {t^m},\\
\sigma_b(t^m\theta)&=& t^m\theta,
\end{eqnarray*}
for $m\in\Z$. From the isomorphism $\sigma_{b}$ of $\tilde{\mathcal{G}}$, it is evident that $\sigma_{b}$ induces a natural $\mathcal{SD}^-$-module structure on any $\mathcal{SD}$-module $V$. As a result, $V$ can be viewed as a $\tilde{\mathcal{G}}$-module using the isomorphism $\Phi$. We can define a new action of $\tilde{\mathcal{G}}$ on $V$ by setting $x\cdot v = \sigma_{b}(x)v$ for all $x\in\tilde{\mathcal{G}}$ and $v\in V$. This new module, denoted by $\mathscr{X}_b(V)$, is referred to as the twisted module of $V$ by $\sigma_{b}$.  It is worth noting that the module $\mathscr{X}_b(V)$ can be seen as a $\mathcal{G}$-module by restricting the action of $\tilde{\mathcal{G}}$ to the subalgebra $\mathcal{G}$. Then
\begin{eqnarray*}
&& L_m\cdot v=\sigma_b(L_m)v=-t^{m}\big( D+mb+\frac{m}{2}\theta\partial_\theta\big)v,
\\&& \label{wq442.2}H_m\cdot v=\sigma_b(H_m)v=t^{m}\big(-2b+\theta\partial_\theta\big)v,
\\&&\label{wq442.3}G_m^+\cdot v=\sigma_b(G_m^+)v=-2t^{m}\big(\theta D+2bm\theta\big)v,
\\&&\label{wq442.4}G_m^-\cdot v=\sigma_b(G_m^-)v=\big(t^{m}\partial_\theta\big)v,
\end{eqnarray*}
for $ m\in\Z,v\in V$.

\end{proof}

\begin{prop}\label{Fun}
For $b\in\mathbb{C}$, there is a functor $\mathscr{X}_b:{\rm Mod}\,\mathcal{SD}\to {\rm Mod}\,\mathcal{G}$ that maps $V$ to $\mathscr{X}_b(V)$, defined by equations (\ref{wq33})-(\ref{wq36}).
\end{prop}
\begin{proof}
The sequence of Lie superalgebras
 $$
\mathcal{G} \stackrel{\iota}{\longrightarrow}  \tilde{\mathcal{G}} \stackrel{\sigma_b}{\longrightarrow}  \tilde{\mathcal{G}} \stackrel{\Phi}{\longrightarrow} \mathcal{SD}^-,
 $$
induces a sequence of  categories of modules for the Lie superalgebras
$$
{\rm Mod}\,\mathcal{SD}^-\to {\rm Mod}\,\tilde{\mathcal{G}}\to {\rm Mod}\,\tilde{\mathcal{G}}\to {\rm Mod}\,{\mathcal{G}}.
$$
It is clear that each $\mathcal{SD}$ serves as a $\mathcal{SD}^-$-module, and each $\tilde{\mathcal{G}}$-module is also a $\mathcal{G}$-module.
It follows from Proposition \ref{lek22} that the functor $\mathscr{X}_b$ can be written as the composition of the following five functors as follows:
$$
\mathscr{X}_b:{\rm Mod}\,\mathcal{SD}\to{\rm Mod}\,\mathcal{SD}^-\to {\rm Mod}\,\tilde{\mathcal{G}}\to {\rm Mod}\,\tilde{\mathcal{G}}\to {\rm Mod}\,{\mathcal{G}}.
$$
\end{proof}

The main theorem in this section can be derived by utilizing Proposition \ref{S} and Proposition \ref{Fun}.

\begin{theo}\label{M}
For $b\in\mathbb{C}$, there exists a functor
$\mathscr{F}_b:{\rm Mod}\,\mathcal{D}\to {\rm Mod}\,\mathcal{G}$ with
$
M\mapsto \mathscr{F}_b(M),
$
where $\mathscr{F}_b(M)=\mathscr{X}_b \mathscr{S}(M)$.
\end{theo}

\section{Properties of functors}
In this section, we will study the properties of irreducibility preservation for these functors and explore the natural isomorphisms that exist between them.

\subsection{Preservation of irreducibility}

\begin{lemm}\label{lemm2.3}Let $V$ be an $\mathcal{SD}$-module, $b\neq0,\frac{1}{2}$. For any $k\in\mathbb{Z}$  and $d\in\Z^*$, we define $T_{k,d}$  as follows:
\begin{eqnarray*}
T_{k,d} &=& \frac{1}{4d^2}(L_{-d}\cdot G^+_{k+d}+L_{d}\cdot G^+_{k-d}-2L_0\cdot G^+_{k})
\end{eqnarray*}
  in $U(\mathcal{G})$. Then, for all   $v \in V$, we have:
$$\frac{1}{b(1-2b)}T_{k,d}\cdot v= t^{k}\theta v.$$
\end{lemm}
\begin{proof}
For any $n,k\in\Z$, we have
 \begin{eqnarray*}
\frac{1}{2}L_n\cdot G^+_{k-n}\cdot v
&=&\Big(t^{n} D+nb t^{n}+\frac{n}{2}t^{n}\theta\partial_\theta\Big)\Big(t^{k-n}\theta D+2(k-n)b t^{k-n}\theta\Big)v
\\&=&\Big(t^{k}\theta D^2+(2b+1)kt^{k}\theta D
+2bk^2t^{k} \theta
\\&&+\Big((-\frac{1}{2}-b)t^{k}\theta D+kb(2b-3)t^{k}\theta\Big)n
+b(1-2b)n^2t^{k}\theta\Big)v.
\end{eqnarray*}
 Choosing $n=-d,0,d$ respectively in the above equation, we obtain that

 $$\frac{1}{ b(1-2b)}T_{k,d}\cdot v=t^{k}\theta v.$$
\end{proof}

\begin{lemm}\label{th36544433}
Let $b\in \mathbb{C}$ and $V$ be an irreducible module over the Weyl superalgebra $\mathcal{SD}$. Then the $\mathcal{G}$-module $\mathscr{X}_b(V)$ is irreducible if  $b\notin\{0,\frac{1}{2}\}$.
\end{lemm}
\begin{proof}

Suppose that  $b\neq0,\frac{1}{2}.$ Choose a nonzero element $v$ from  $V_{\bar0}$. For any   $k\in\mathbb{Z}$, and nonzero integer $d$, we can use the linear operator  $T_{k,d}$  defined in Lemma \ref{lemm2.3} to obtain the following results:
\begin{eqnarray}\label{322}
&&\frac{1}{b(1-2b)}T_{k,d}\cdot v=\overline{t^{k}v}\in U(\mathcal{G})v,
\\&&\frac{1}{b(1-2b)}G_{0}^-\cdot T_{k,d}  \cdot v=t^{k} v\in U(\mathcal{G})v.
\end{eqnarray}
For $n\in\N$, one has
\begin{eqnarray}\label{266}
(-L_0)^n\cdot v= D^nv\in U(\mathcal{G})v.
\end{eqnarray}
From \eqref{322}-\eqref{266}, we can conclude that $U(\mathcal{G})v=\mathscr{X}_b(V).$ This implies that $\mathscr{X}_b(V)$ is an irreducible $\mathcal{G}$-module.
\end{proof}

\begin{lemm}\label{lemm200}
Let   $M$ be an irreducible module over the Weyl algebra $\mathcal{D}$.
 \begin{itemize}
\item[\rm (i)]  The $\mathcal{G}$-module $\mathscr{F}_{0}(M)$ is reducible  if and only if   $M$ is isomorphic to the $\mathcal{D}$-module $\C[t^{\pm1}]$;
\item[{\rm (ii)}]
If $M=\C[t^{\pm1}]$, then $\mathscr{F}_{0}(\C[t^{\pm1}])$ has a maximum submodule $\C$. Consequently, the factor module $\mathscr{F}_{0}(\C[t^{\pm1}])/\C$ is an irreducible $\mathcal{G}$-module.
 \end{itemize}
\end{lemm}
\begin{proof}
{\rm (i)} Let $\overline{v}$ be any nonzero element in  $\overline{M}$. Now we consider the action of $D$ on $\overline{v}$.
\begin{case}\label{case1}
The
action of $D$ on $\overline{v}$ is nilpotent.
\end{case}
It is evident that there is a non-zero element $\overline{v}\in\overline{M}$ such that $\overline{Dv}=0$. This implies that $\overline{D({t}^k{v})}=k\overline{t^k{v}}$. Therefore, we can conclude that $V$ is isomorphic to the $\W$-module $\C[t^{\pm1}]\oplus\overline{\C[{t}^{\pm1}]}$. It's easy to verify that the $\mathcal{G}$-module $\mathscr{F}_{0}(M)$ has a one-dimensional submodule, namely, $\C$, and hence it's reducible.

\begin{case}
The
action of $ D$ on $\overline{v}$ is non-nilpotent.
\end{case}
We can choose a non-zero $
\overline{w}=\overline{D^kv}\in \overline{M},
$ where $k\in\Z_+$.  For $m\in\Z,d\in\Z^*$, we define
$$
Q_{m,d}=\frac{2}{d^2}(L_{-d}\cdot L_{m+d}+L_{d}\cdot L_{m-d}-2
L_0\cdot L_{m}).$$
By the similar calculation in Lemma \ref{lemm2.3}, we can get   $$Q_{m,d}\cdot\overline{w}=\overline{t^m w}\in U(\mathcal{G}){v},
$$ where $m\in\Z,d\in\Z^*$.  Using this, one gets
$$G_{0}^-\cdot(\overline{t^mw})=t^mw\in U(\mathcal{G}){v}.$$
 For $i\in\N$, one has
$$
(-L_0)^i\cdot \overline{w}= \overline{D^iw}\in U(\mathcal{G}){v}.
$$
  Thus, $\mathscr{F}_{0}(M)=U(\mathcal{G}){v}$. In other words,  $\mathscr{F}_{0}(M)$ is an irreducible $\mathcal{G}$-module  in this case.

{\rm (ii)}
It is clear that $\big(\C[t^{\pm1}]\oplus\overline{\C[{t}^{\pm1}]}\big)/ \C=\mathrm{span}_{\C}\{t^n, \overline{t^k} \mid n \in\Z^*,k\in\Z\}$. For
$n\in\Z^*,k\in\Z$, $m\in\Z$, we have the nontrivial relations
\begin{eqnarray*}
&&L_m\cdot t^n=-t^m D(t^n)=-nt^{m+n},
\\&&L_m\cdot\overline{t^k}=-(t^m D+\frac{m}{2}t^m\theta\partial_\theta)\overline{t^k}=-(k+\frac{m}{2})(\overline{t^{m+k}}),
\\
&&
H_m\cdot \overline{t^k}=(t^m\theta\partial_\theta)\overline{t^k}=\overline{t^{m+k}},
\\&& G_m^+\cdot t^n=(-2t^{m}\theta D)t^n=-2n\overline{t^{m+n}},
\\&& G_m^-\cdot\overline{t^k}=(t^{m}\partial_\theta)\overline{t^k}=t^{m+k}.
\end{eqnarray*}
This implies  that $\big(\C[t^{\pm1}]\oplus\overline{\C[{t}^{\pm1}]}\big)/ \C$ is an irreducible $\mathcal{G}$-module.
\end{proof}

Suppose that $M$ is an irreducible $\mathcal{D}$-module, we define
$$
\overline{\mathscr{F}_{0}(M)}=\begin{cases}
&\mathscr{F}_{0}(M)/\C,\quad \text{if}\quad  M\cong \C[t^{\pm1}];\\
&\mathscr{F}_{0}(M),\quad\quad\text{otherwise}.
\end{cases}
$$

Assume that  $\sigma\in \mathrm{Aut}(\mathcal{G})$ with
\begin{eqnarray*}
\sigma(L_m)=L_m,\ \sigma(H_m)=-H_m,\ \sigma(G^+_m)=-2G^-_m,
\ \sigma(G_m^-)=-\frac{1}{2}G_m^+,
\end{eqnarray*}
where $m\in\Z.$

Let   $V$ be any  module of $\mathcal{G}$. We can
make $V$ into another $\mathcal{G}$-module, by defining the new action of $\mathcal{G}$ on $V$ as
$$x\ast v=\sigma(x)v,$$ where $x\in \mathcal{G}$. We denote the new module by
$V^\sigma$, and call it as the twisted module of $V$ by $\sigma$.

\begin{lemm}\label{lemm201}
Let $M$ be an irreducible $\mathcal{D}$-module. Then the following holds.
\begin{itemize}
    \item[\rm (i)]  $\mathscr{F}_{\frac{1}{2}}(M)$ is an irreducible $\mathcal{G}$-module if and only if $D(M) = M$.
    \item[\rm (ii)]If $D(M) = M$, then $\mathscr{F}_{\frac{1}{2}}(M)\cong\Pi\circ \overline{\mathscr{F}_{0}(M)}^\sigma$, where $\Pi$ is the parity change functor for $\mathcal{G}$-modules and
    $\overline{\mathscr{F}_{0}(M)}$ as   $\mathcal{G}$-modules.
\end{itemize}

\end{lemm}
\begin{proof}
{\rm(i)}
Consider   $M\oplus\overline{M}$ as a $\Z_2$-graded modules over $\W$. Let $v\in M$ and $\overline{v} \in \overline{M}$. For
any $m\in\Z$,  $\mathcal{G}$-module $\mathscr{F}_{\frac{1}{2}}(M)$ is defined by

\begin{eqnarray*}
&& L_m\cdot v=-\big(t^{m} D+\frac{m}{2} t^{m}\big)v,
\\&&L_m\cdot\overline{v}=-\big(t^{m} D+m t^{m}\big)\cdot\overline{v}=-\overline{D t^{m}v},
\\&& H_m\cdot v=-t^mv,
\\&&G_m^{+}\cdot v=-2\big(t^{m} D\theta+
mt^m\theta\big)\cdot v=-2\overline{D t^{m} v},
\\&&G_m^{-}\cdot\overline{v}= t^{m}v, \\
&&\ H_m\cdot \overline{v}=G_m^{+}\cdot\overline{v}=G_m^{-}\cdot v=0.
\end{eqnarray*}
We see that $(\mathscr{F}_{\frac{1}{2}}(M))_{\bar0}\oplus D(\mathscr{F}_{\frac{1}{2}}(M))_{\bar1}$ is a $\mathcal{G}$-submodule, namely, $M=D(M)$. For $k\in\Z$  and $v\in M$, we have $-H_k\cdot v=t^{k}v\in U(\mathcal{G})v$,  and $-\frac{1}{2}G_k^+\cdot v=\overline{Dt^{k}v}\in U(\mathcal{G})v$.
For any $i\in\mathbb{N}$ and $v\in M$, we have $(-L_0)^i\cdot v= D^iv\in U(\mathcal{G})v$.  Evidently, $(\mathscr{F}_{\frac{1}{2}}(M))_{\bar 0}\oplus D  (\mathscr{F}_{\frac{1}{2}}(M))_{\bar 1}$ is
an irreducible $\mathcal{G}$-submodule of $\mathscr{F}_{\frac{1}{2}}(M)$.

{\rm(ii)}  If $\mathscr{F}_{0}(M)$ is  irreducible,   we define the following linear   map
 \begin{eqnarray*}
\phi:\quad (\mathscr{F}_{\frac{1}{2}}(M))_{\bar0}\oplus D(\mathscr{F}_{\frac{1}{2}}(M))_{\bar1}\quad&\longrightarrow& \Pi(\mathscr{F}_{0}(M)^\sigma)
\\ v&\longmapsto&\overline{v}
\\  \overline{Dv}&\longmapsto& v,
\end{eqnarray*}
where $v\in M,\overline{v}\in\overline{M}$.
For $m\in\Z$, one can check that
 $\phi$ is a $\mathcal{G}$-module isomorphism.

If $\mathscr{F}_{0}(M)$ is  reducible, from Case $1$ in Lemma \ref{lemm200} (i),  one gets $$\overline{\mathscr{F}_{0}(M)}=\big(\C[t^{\pm1}]\oplus\overline{\C[{t}^{\pm1}]}\big)/ \C.$$ We see that the
following linear map $\psi$ is a $\mathcal{G}$-module isomorphism, where
\begin{eqnarray*}
\psi:\quad (\mathscr{F}_{\frac{1}{2}}(M))_{\bar0}\oplus D(\mathscr{F}_{\frac{1}{2}}(M))_{\bar1}\quad&\longrightarrow& \overline{\mathscr{F}_{0}(M)}^\sigma
\\t^{m}&\longmapsto&\overline{t^m}
\\  \overline{D t^n}&\longmapsto& t^n,
\end{eqnarray*}
for $ m\in\Z,n\in\Z^*.$
\end{proof}

The main theorem in this paper can be derived by utilizing  Lemmas \ref{lemm2.3}-\ref{lemm201}.

\begin{theo}\label{M1}
For $b\in\mathbb{C}$, the functor
$\mathscr{F}_b$  preserves irreducible $\mathcal{D}$-modules  if and only if $b\notin\{0,\frac{1}{2}\}$.
\end{theo}

\subsection{Isomorphisms
}\label{ir4}

\begin{lemm}\label{lemm20011}
Let $ b_1, b_2\in \C, b_1\notin\{0,\frac{1}{2}\}$ and $M_1,M_2$ be irreducible $\mathcal{D}$-modules.
 Then  $\mathscr{F}_{b_1}(M_1)\cong \mathscr{F}_{b_2}(M_2)$, as   $\mathcal{G}$-modules,  if and only if
 $b_1=b_2$ and   $M_1\cong M_2$, as $\mathcal{D}$-modules.
\end{lemm}
\begin{proof}

 The sufficiency is evident, so we only need to prove the necessity. Suppose that $\psi: \mathscr{F}_{b_1}(M_1) \rightarrow \mathscr{F}_{b_2}(M_2)$ is a $\mathcal{G}$-module isomorphism. For any $k\in \mathbb{Z}$, $d\in\mathbb{Z}^*$, and $0\neq v\in M_1$, we have $\psi(T_{k,d}\cdot v)=T_{k,d}\cdot\psi(v)$ by applying Lemma \ref{lemm2.3}. This, together with the same lemma, yields
\begin{eqnarray}\label{eq401}
b_1(1-2b_1)\psi(\overline{t^{k} v})=b_2(1-2b_2)\overline{t^{k} \psi(v)}.
\end{eqnarray}
Note that $b_1\neq 0$ and $b_1\neq\frac{1}{2}$. By Proposition \ref{S}, we see that the action of $\partial_\theta$ on $v\in M_1$ is trivial. Then, using the fact that $\psi (\frac{G_0^-\cdot (\overline{t^{k} v})}{b_2(1-2b_2)})=G_0^-\cdot \frac{\overline{t^k\psi(v)}}{b_1(1-2b_1)}$, we get
\begin{eqnarray}\label{eq411}
\psi\Big(\frac{t^{k} v}{b_2(1-2b_2)}\Big)=\frac{t^{k} \psi(v)}{b_1(1-2b_1)}.
\end{eqnarray}
Setting $k=0$ in \eqref{eq411}, we obtain
\begin{eqnarray}\label{eq422}
b_1(1-2b_1)=b_2(1-2b_2).
\end{eqnarray}
Substituting \eqref{eq422} into \eqref{eq401} and \eqref{eq411}, we immediately obtain
$$\psi(\overline{t^{k}  v})=\overline{t^{k}  \psi(v)} \quad \mathrm{and} \quad \psi(t^{k}  v)=t^{k} \psi(v),$$
where $k\in\mathbb{Z}$. For any $i\in\mathbb{N}$ and $v\in V$, we have $\psi(L_0^i\cdot v)=L_0^i\cdot\psi(v)$, which implies $\psi( D^iv)= D^i\psi(v)$. Therefore, we conclude that $M_1\cong M_2$.

Now, for any $m\in\mathbb{Z}$ and $0\neq v\in M_1$, from $\psi(L_m\cdot v)=L_m\cdot\psi(v)$, it is easy to check that $b_1=b_2$.
\end{proof}

\begin{lemm}\label{lemm201111}
Let  $M_1$ and $M_2$ be irreducible $\mathcal{D}$-modules.
 Then
  $\mathscr{F}_{0}(M_1)\cong \mathscr{F}_{0}(M_2)$ as   $\mathcal{G}$-modules if and only if
    $M_1\cong M_2$ as $\mathcal{D}$-modules.
\end{lemm}
\begin{proof}
The sufficiency of the conditions is clear. Suppose that $\psi:\mathscr{F}_{0}(M_1)\cong \mathscr{F}_{0}(M_2)$ is an  isomorphism. For any $\overline{v}\in \overline{M_1}$, we note that $-L_0\cdot  \overline{v}= \overline{Dv}$.
We consider the action of $D-k$ on $\overline{v}$.

If $(D-k)\cdot\overline{v}=0$ for some $k\in\Z$ and a nonzero $\overline{v}\in \overline{M_1}$, then $ D(\overline{t^{-k}v})=0$ where $t^{-k}v\neq0$.
According to  Case $1$ in the proof of Lemma \ref{lemm200}, we know that the $\W$-module $M_1\oplus \overline{M_1} \cong \C[t^{\pm1}]\oplus\overline{\C[{t}^{\pm1}]}$.
Similarly, using $(D-k)\psi(\overline{v})=0$ in $\overline{M}_2$, we deduce that $M_2\oplus\overline{M}_2\cong\C[t^{\pm1}]\oplus \overline{\C[{t}^{\pm1}]}$. Thus $M_1\cong M_2$ in this case.

Consider that $D-k$ is injective on both $\overline{M_1}$ and $\overline{M_2}$ for all $k\in\Z$. By  Proposition \ref{S}, we see that any $v\in M_1$ can be annihilated by $\partial_\theta$.
For any $v\in M_1, k\in\Z$, we have
\begin{eqnarray}\label{908e}
\nonumber \psi(-\frac{1}{2}G_k^+\cdot v)&=&\psi(\overline{t^{k}Dv})=\psi(\overline{(D-k)t^{k}v})
\\&=&\psi((-L_0-k)\cdot (\overline{t^{k}v}))=(D-k)\cdot\psi(\overline{t^{k}v}).
\end{eqnarray}
Inserting \eqref{908e} into
 $\psi(G_k^+\cdot v)=G_k^+\cdot\psi(v)$,
 it is easy to check  $\psi(\overline{t^{k}v})=\overline{t^{k}\psi(v)}$ for $k\in\Z$.
 Then by $\psi(G_{0}^-\cdot(\overline{t^{k}v}))=G_{0}^-\cdot(\overline{t^{k}\psi(v)})$,
we see that
$\psi(t^{k}v)=t^{k}\psi(v)$ for all $v\in M_1$ and $k\in\Z$.   For $i\in\N,v\in M_1$,  from   $\psi(L_0^i\cdot v)=L_0^i\cdot\psi(v)$,  one has  $\psi( D^iv)= D^i\psi(v)$.  Then we also have $M_1\cong M_2$.
\end{proof}

From  Lemmas  \ref{lemm201}, \ref{lemm20011},  \ref{lemm201111}, we  obtain the following isomorphism theorem.

\begin{theo}\label{th4333}
Let $b_1, b_2\in \C$, $M_1,M_2$ be irreducible $\mathcal{D}$-modules. Then $\mathscr{F}_{b_1}(M_1)\cong \mathscr{F}_{b_2}(M_2)$
as   $\mathcal{G}$-modules if and only if  one of the following holds:
\begin{itemize}
    \item[{\rm (i)}] $b_1=b_2,$ $M_1\cong M_2;$
    \item[{\rm (ii)}] $(b_1,b_2)=(\frac{1}{2},0),$ $M_1\cong M_2$, and $M_1= D  (M_1);$
    \item[{\rm (iii)}] $(b_1,b_2)=(0,\frac{1}{2}),$ $M_1\cong M_2$, and $M_2= D  (M_2)$.

\end{itemize}

\end{theo}

\section{Applications and examples}

This section applies the functors $\mathscr{F}_b$ to retrieve the functors that were originally presented in \cite{CDLP} for N=1 super-Virasoro algebras and construct several irreducible  modules for N=2  superconformal algebras.

\subsection{Applications}

Suppose that  $\epsilon\in\{0,\frac{1}{2}\}$. In the N=2 superconformal algebra $\hat{\mathcal{G}}[{\epsilon}]$, let $G_p=-(\frac{1}{2}G_p^++G_p^-)$. Then the Lie subalgebra generated by
 $\{L_n,G_p,C \mid  n\in\mathbb{Z}, p\in \mathbb{Z}+\epsilon\}$ is isomorphic to N=1 superconformal (super-Virasoro) algebra $\hat{\mathcal{S}}[\epsilon]$ satisfying
\begin{equation*}\label{def1.1}
		\aligned
		[L_m,L_n]&= (m-n)L_{m+n}+\frac{m^{3}-m}{12}\delta_{m+n,0}C,\\
		[G_p,G_q]&= 2L_{p+q}+\frac{4p^{2}-1}{12}\delta_{p+q,0}C,\\
		[L_m,G_q]&= \left(\frac{m}{2}-q\right)G_{m+q},\\
	[\hat{\mathcal{S}}[\epsilon],C]&=0,
		\endaligned
	\end{equation*}  where $m,n\in\mathbb{Z}$, $p,q\in\mathbb{Z}+\epsilon$. It is known that $\hat{\mathcal{S}}[0]$  is called the N=1 Ramond algebra and
$\hat{\mathcal{S}}[\frac{1}{2}]$ is called
the  N=1 Neveu-Schwarz algebra. Moreover, $\hat{\mathcal{S}}[\frac{1}{2}]$  is isomorphic to the subalgebra of $\hat{\mathcal{S}}[{0}]$ spanned by
$\{L_{m}\mid m\in2\Z\}\cup\{G_r\mid r\in2\Z+1\}\cup\{C\}$. Please refer to \cite{S}, \cite{CL}, \cite{CLL}, \cite{CDLP},\cite{IK0}, \cite{IK},\cite{IK1},\cite{LPX},\cite{LPX1},\cite{XS} and their respective references.
To simplify notation, we denote
$$ \mathcal{S}[{\epsilon}]=\hat{\mathcal{S}}[{\epsilon}]/\C C.
$$
By Lemma \ref{lemm2.5} and Lemma  \ref{lek22}, we  can define   the   $\mathcal{G}[{\frac{1}{2}}]$-modules $\mathscr{F}_b^\prime(M)=M\oplus\overline{M}$
     as follows
\begin{eqnarray*}
&& L_m\cdot v=-\big(t^{m} D+(m-1)b t^{m} +\frac{m+1}{2}t^{m}\theta\partial_\theta\big)v,\\&&H_m\cdot v=t^{m}\big(-2b+\theta\partial_\theta\big)v,
\\&&G_p^+\cdot v=-2t^{p-\frac{1}{2}}\big(\theta D+2b(p-\frac{1}{2})\theta\big)v,
\\&&G_p^-\cdot v=\big(t^{p+\frac{1}{2}}\partial_\theta\big)v,
\end{eqnarray*}
where $m\in\Z,p\in\Z+\frac{1}{2},v\in M\oplus\overline{M}$.

We define the restriction functor $R_{\epsilon}:{\rm Mod}\,\mathcal{G}[{\epsilon}]\to {\rm Mod}\,\mathcal{S}[{\epsilon}]$ induced by the embedding $\mathcal{S}[{\epsilon}]\hookrightarrow\mathcal{G}[\epsilon]$. Using this, we can obtain a new functor $H_{\epsilon,b}=R_{\epsilon} \circ \mathscr{F}_b:{\rm Mod}\,\mathcal{D}\to {\rm Mod}\,\mathcal{S}[\epsilon]$ that maps $M$ to $H_{\epsilon,b}(M)$, where $H_{\epsilon,b}(M)=M \oplus \overline{M}$ is defined by
\begin{eqnarray*}
L_m\cdot v&=&-t^{m}\left(D+(m-2\epsilon)b+\frac{m+2\epsilon}{2}\theta\partial_\theta\right)v, \\
G_p\cdot v &=&t^{p-\epsilon}\left(\theta D+2(p-\epsilon)b\theta-t^{2\epsilon}\partial_\theta\right)v
\end{eqnarray*}
 for $m\in\mathbb{Z}$, $p\in\mathbb{Z}+\epsilon$,  and $v\in M\oplus \overline{M}$.
 To visually view the process of the  construction  of
 $\mathcal{S}[\epsilon]$-modules $H_{\epsilon,b}(M)$, we give the following  diagram.
\begin{displaymath}
\xymatrix{\mathscr{F}_b(M)\ar@{.>}[rr]^{R_0}\ar[dr]_{\delta}&&H_{\epsilon,b}(M)\\
 &\mathscr{F}^\prime_{b}(M)\ar[ur]_{R_{\frac{1}{2}}}}
\end{displaymath}

 The functor obtained here coincides with those studied in detail in reference \cite{CDLP}.

\begin{theo}[\cite{CDLP}]\label{ma}

Let  $\epsilon\in\{0,\frac{1}{2}\},b,b_1,b_2\in\mathbb{C}$ and $b_1\notin \{0,\frac{1}{2}\}$. Then
\begin{itemize}
    \item[{\rm (1)}] The functor $H_{\epsilon,b}$ preserves irreducible modules if and only if $b\notin\{0,\frac{1}{2}\}$.
\item[{\rm (2)}] $H_{\epsilon, b_1}\cong H_{\epsilon,b_2}$ if and only if  $b_1=b_2$.

\item[{\rm (3)}]  Let $M$ be an irreducible $\mathcal{D}$-module. If $ b\in\{0,\frac{1}{2}\}$,
then $H_{\epsilon,b}(M)$ is irreducible if and only if one of the following holds:
 \begin{itemize}
\item[{\rm (i)}] $b=0$ and  $M\ncong\C[t^{\pm1}]$;

\item[{\rm (ii)}] $b=\frac{1}{2}$ and $M=D(M)$.
\end{itemize}
\item[{\rm (4)}] Let $M_1,M_2$ be irreducible $\mathcal{D}$-modules.
Then $H_{\epsilon,\frac{1}{2}}(M_1)\cong H_{\epsilon,0}(M_2)$ as $\mathcal{S}[\epsilon]$-modules if and only if   $M_1\cong M_2$, and $M_1= D  (M_1)$.
\end{itemize}
\end{theo}

\subsection{Examples}
In this subsection, we utilize the functors $\mathscr{F}_b$ and their restrictions  to rediscover several  known irreducible $\mathcal{G}$-modules.  Additionally, we construct lots of new irreducible $\mathcal{G}$-modules.
\subsubsection{Intermediate series modules}
Let $\alpha\in\C[t^{\pm1}]$. Set  $\tau= D-\alpha$ in Lemma \ref{211} (i),
we get the irreducible $\mathcal{D}$-module
$M_{\alpha}=\mathcal{D} /\mathcal{D}\tau$, which has a basis $\{t^k\mid k\in\Z\}$ with the actions
\begin{eqnarray*}
D\cdot t^n=t^n(\alpha+n),\quad t^m\cdot t^n=t^{m+n}, \ \forall m,n\in\Z.
\end{eqnarray*}

For $b\in\C$,   from  \eqref{wq33}-\eqref{wq36} we obtain    $\mathcal{G}$-module
$M_{\alpha,b}=\mathscr{F}_b(\C[t^{\pm1}])=\C[t^{\pm1}]\oplus  \overline{\C[t^{\pm1}}]$
 with the following  actions:
\begin{eqnarray*}
 && L_m\cdot t^n=-(\alpha+n+mb)t^{m+n},
 \\&&  L_m \cdot\overline{t^n}=-(\alpha+n+m(b+\frac{1}{2}))(\overline{t^{m+n}}),
  \\&&
 H_m\cdot t^n=-2bt^{m+n},
 \\&&   H_m \cdot\overline{t^n}=(1-2b)\overline{t^{m+n}},
 \\&&
 G_m^+ \cdot t^n=-2(\alpha+n+2mb)\overline{t^{m+n}},\\&&
  G_m^+ \cdot\overline{t^n}= G_m^- \cdot t^n=0,
  \\&&
  G_m^-\cdot\overline{t^n}=t^{m+n},
 \end{eqnarray*}
 for $m,n\in\Z$.

 \begin{coro}[\cite{LPX0}]Let $\alpha,b\in\C$. Then
\begin{itemize}
    \item[{\rm (1)}]   $M_{\alpha,b}$  is isomorphic to the intermediate series weight modules  for $\mathcal{G}$.
 \item[{\rm (2)}]$M_{\alpha,b}$ is irreducible except that $\alpha\in\Z,b=0$ or $\alpha\in\Z,b=\frac{1}{2}$.
\end{itemize}

 \end{coro}

 \begin{coro}
Let $\alpha\in\C[t^{\pm1}]\setminus\C$. Then  $\M_{\alpha,b}$ is  irreducible  if $b\neq\frac{1}{2}$.

 \end{coro}






\subsubsection{$U(\mathbb{C}L_0)$-free modules }
For $\lambda\in\C^*$, we recall the irreducible $\mathcal{D}$-module $\Omega(\lambda)$ defined in Lemma \ref{211}, which has a basis $\{D^n\mid n\in\N\}$, and the $\mathcal{D}$-actions  are defined as follows
\begin{eqnarray*}
&&t^m\cdot D^n=\lambda^m( D-m)^n, \quad  D\cdot D^n= D^{n+1}
\end{eqnarray*}
for $m\in\Z, n\in\N$.
For   $b\in\C$,
based on   \eqref{wq33}-\eqref{wq35}, we have a family of $\mathcal{G}$-modules  $\Omega(\lambda,b)=\mathscr{F}_b(\Omega(\lambda))
$ with the nontrivial actions:
\begin{eqnarray*}
 && L_m  \cdot D^n=-\lambda^m\big( D+m(b-1)\big)( D-m)^{n},
 \\&& L_m \cdot\overline{D}^n=-\lambda^m\big( \overline{D}+m(b-\frac{1}{2})\big)(\overline{D}-m)^{n},
 \\&&
 H_m \cdot D^n=-2b\lambda^m( D-m)^{n},
 \\&& H_m\cdot \overline{D}^n=(1-2b)\lambda^m(\overline{D}-m)^{n},
 \\&&
 G_m^+ \cdot D^n=-2\lambda^m\big(\overline{D}+(2b-1)m\big)(\overline{D}-m)^{n},
 \\&& G_m^- \cdot\overline{D}^n=\lambda^m( D-m)^{n},
 \end{eqnarray*}
for $m\in\Z, n\in\N$.

By means of direct calculation, the $U(\mathbb{C}L_0)$-free module that was constructed and analyzed in \cite{YYX2} is isomorphic to $\Omega(\lambda,b)$. According to Theorem \ref{M}, we can deduce the following corollary, originally presented in \cite{YYX2}.
\begin{coro}
$\Omega(\lambda,b)$ is irreducible if only if $b\neq \frac{1}{2}$.
\end{coro}

\subsubsection{Fraction modules}\label{65ty}

Let $n\in\N,\alpha=(\alpha_0,\alpha_1,\ldots,\alpha_n)\in\C^{n+1},\beta=(\beta_0,\beta_1,\ldots,\beta_n)\in\C^{n+1}$
with $\beta_0=0$ and $\beta_i\neq \beta_j$ for all $i\neq j$. Set
$\tau=\frac{d}{dt}-\sum_{i=0}^n\frac{\alpha_i}{t-\beta_i}$ in Lemma   \ref{211} (i).
Then we get  the irreducible $\mathcal{D}$-module
\begin{eqnarray*}
M(\alpha,\beta)=\mathcal{D}/\big(\mathcal{D}\cap(\C(t)[ D]\tau)\big)
 \subset\C[t,(t-\beta_i)^{-1}\mid i=0,1,\ldots,n]
\end{eqnarray*}
with the actions:
\begin{eqnarray*}
&&\frac{d}{dt}\cdot f(t)=\frac{d}{dt} (f(t))+f(t)\sum_{i=0}^n\frac{\alpha_i}{t-\beta_i},
\quad
t^m\cdot f(t)=t^m f(t),
\quad \partial_\theta\cdot  f(t)=0,
\end{eqnarray*} where
  $f\in M(\alpha,\beta), m\in\Z.$

 For any $b\in\C$,
  from  \eqref{wq33}-\eqref{wq36} we  obtain a new family of   $\mathcal{G}$-module  $M_b(\alpha,\beta)=\mathscr{F}_b(M(\alpha,\beta))$ with the following nontrivial actions:
\begin{eqnarray*}
  &&L_m\cdot f(t)=-t^{m+1}\frac{d}{dt} (f(t))-t^{m+1}f(t)\sum_{i=0}^n\frac{\alpha_i}{t-\beta_i}-mb t^mf(t),\\
 &&L_m\cdot  \overline{f(t)}=-\overline{t^{m+1}}\frac{d}{dt} (\overline{f(t)})-\overline{t^{m+1}f(t)}\sum_{i=0}^n\frac{\alpha_i}{\overline{t}-\beta_i}-m(b+\frac{1}{2}) \overline{t^mf(t)},\\
 &&H_m\cdot f(t)=-2bt^mf(t),
 \\&&H_m\cdot \overline{f(t)}=(1-2b)\overline{t^{m} f(t)},
 \\&&G_m^+\cdot f(t)=-2\Big(\overline{t^{m+1}} \frac{d}{dt} (\overline{f(t)})+\overline{t^{m+1}f(t)}\sum_{i=0}^n\frac{\alpha_i}{\overline{t}-\beta_i}+2mb \overline{t^m f(t)}\Big),
 \\&&G_m^-\cdot\overline{f(t)}=t^mf(t).
 \end{eqnarray*}
 By Theorem \ref{M},
 we  know  that   the fraction module  $M_b(\alpha,\beta)$ is  irreducible if and only if  $b\neq\frac{1}{2}$.




\subsubsection{Special  degree $n$ modules}\label{64ty}
For any $n\in\N$, some degree $n$ irreducible elements in $\C(t)[ D]$ were determined in  Lemma 16 of \cite{LZ}.

For any $n\in\N$, choose $\tau=\left(\frac{d}{dt}\right)^n-t$ in Lemma \ref{211} (i). Then it is easy to get  the irreducible $\mathcal{D}$-module
$M=\mathcal{D}/\big(\mathcal{D}\cap(\C(t)[ D]\tau)\big)$
with a basis
$$\left\{t^k,\left(\frac{d}{dt}\right)^m\mid k\in\Z, m=0,1, \ldots, n-1\right\}.$$ The actions of $\mathcal{D}$ are defined as

\begin{eqnarray*}
  &&t^k\cdot (t^i\left(\frac{d}{dt}\right)^m)=t^{k+i}\left(\frac{d}{dt}\right)^m\quad \mathrm{for} \ k,i\in\Z, 0\leq m\leq n-1,\\
  && \frac{d}{dt}\cdot (t^i\left(\frac{d}{dt}\right)^m)=it^{i-1}\left(\frac{d}{dt}\right)^m+t^i\left(\frac{d}{dt}\right)^{m+1}\quad \mathrm{for} \ i\in\Z, 0\leq m< n-1,\\
 && \frac{d}{dt}\cdot (t^i\left(\frac{d}{dt}\right)^{n-1})=it^{i-1}\left(\frac{d}{dt}\right)^{n-1}+t^{i+1}\quad \mathrm{for} \ i\in\Z.
 \end{eqnarray*}

For   $b\in\C$, we obtain the $\mathcal{G}$-module $\mathscr{F}_b(M)$
with the nontrivial actions as follows
\begin{eqnarray*}
  &&L_k\cdot (t^i(\frac{d}{dt})^m)=-(i+b k)t^{k+i}(\frac{d}{dt})^m-t^{k+i+1}(\frac{d}{dt})^{m+1},\\
 &&L_k\cdot (t^i(\frac{d}{dt})^{n-1})=-(i+b k)t^{k+i}(\frac{d}{dt})^{n-1}-t^{k+i+2},\\
 &&L_k\cdot (\overline{t^i(\frac{d}{dt})^m})=-(i+(b+\frac{1}{2}) k)\overline{t^{k+i}(\frac{d}{dt})^m}-\overline{t^{k+i+1}(\frac{d}{dt})^{m+1}},\\
 &&L_k\cdot (\overline{t^i(\frac{d}{dt})^{n-1}})=-(i+(b+\frac{1}{2}) k)\overline{t^{k+i}(\frac{d}{dt})^{n-1}}-\overline{ t^{k+i+2}},
 \\&&
H_k\cdot (t^i(\frac{d}{dt})^m)= -2bt^{k+i}(\frac{d}{dt})^m,\\&&
H_k\cdot (t^i(\frac{d}{dt})^{n-1})= -2bt^{k+i}(\frac{d}{dt})^{n-1},
\\&& H_k\cdot (\overline{t^i(\frac{d}{dt})^{m}})=(1-2b)\overline{t^{k+i}(\frac{d}{dt})^m},
 \\&& H_k\cdot (\overline{t^i(\frac{d}{dt})^{n-1}})=(1-2b)\overline{t^{k+i}(\frac{d}{dt})^{n-1}},
 \\
 &&G_k^+\cdot( t^i(\frac{d}{dt})^m)=-2\Big((i+2bk)\overline{t^{k+i}(\frac{d}{dt})^m}+\overline{t^{k+i+1}(\frac{d}{dt})^{m+1}}\Big),
 \\&&
 \ G_k^+\cdot ( t^i(\frac{d}{dt})^{n-1})=-2\Big((i+2bk)\overline{t^{k+i}(\frac{d}{dt})^{n-1}}+\overline{t^{k+i+2}}\Big),
 \\&&
 G_k^-\cdot (\overline{t^i(\frac{d}{dt})^m})=t^{k+i}(\frac{d}{dt})^m,
 \\&&G_k^-\cdot (\overline{t^i\left(\frac{d}{dt}\right)^{n-1}})=t^{k+i}\left(\frac{d}{dt}\right)^{n-1},
 \end{eqnarray*}
  where  $k,i\in\mathbb{Z},0\leq m< n-1$.

  By Theorem   \ref{M},
 we  know  that   the special degree $n$ modules   $\mathscr{F}_b(M)$ is  irreducible if and only if  $b\neq\frac{1}{2}$.

\section*{Acknowledgements}
This work was supported by the National Natural Science Foundation of China (Grant Nos. 12171129, 12071405)
and  Fujian Alliance of Mathematics (Grant Nos. 2023SXLMMS05).

\end{document}